\documentclass[letterpaper]{amsart}
\usepackage{amsmath,amsfonts,amsthm,amssymb}
\usepackage[all]{xy}
\usepackage{enumerate}
\usepackage{graphicx}
\usepackage{mathtools}
\usepackage{hyperref}
\usepackage{bbm}
\usepackage{xcolor}
\usepackage[T1]{fontenc}
\usepackage{paralist}

\newcommand{\RR}{\mathbb{R}}

\newcommand{\NN}{\mathbb{N}}
\newcommand{\EE}{\mathbb{E}}

\newcommand{\dd}{\mathrm{d}}

\theoremstyle{definition}
\newtheorem{thm}{Theorem}[section]
\newtheorem{lma}[thm]{Lemma}

\newtheorem{prop}[thm]{Proposition}
\newtheorem{rmk}[thm]{Remark}

\newtheorem{definition}[thm]{Definition}

\allowdisplaybreaks

\title[Centroids for weak optimal transport]{Centroids for weak optimal transport with barycentric cost}

\author{Friedemann Krannich}
\date{September 15, 2026}
\address{Department of Mathematics, University of Toronto, Bahen Centre, 40 St. George St., Toronto, ON, M5S 2E4, Canada}
\email{f.krannich@mail.utoronto.ca}
\thanks{The author received support from the Department of Mathematics at the University of Toronto, the Canada Research Chairs program (Grant \#CRC-2020-00289) and the Natural Sciences and Engineering Research Council of Canada Discovery Grant (Grant \#RGPIN-2020-04162 and Grant \#RGPIN-2024-06235).}
\subjclass{49Q22, 60G42, 49J55}
\keywords{Weak optimal transport, Wasserstein barycenter, Convex order, Multi-marginal optimal transport}

\begin{document}
	
\begin{abstract} 
	We generalize the theory of Wasserstein barycenters between N given measures from the Wasserstein distance to the weak optimal transport problem with barycentric cost. We find a multi-marginal formulation and its dual problem. The multi-marginal problem is new in itself; contrary to the classical multi-marginal optimal transport problem, the first M marginals of its competitors are not fixed but instead only need dominate the first M given measures in convex order, while the last N-M marginals of its competitors are dominated in convex order by the last N-M given measures. 
\end{abstract}

\maketitle

\section{Introduction}
\subsection{From classical to weak optimal transport}
Let $\Omega\subset\RR^d$ be a compact and convex set.
The \textit{classical optimal transport problem} describes the minimal cost of transporting a probability measure $\mu_1\in\mathcal P(\Omega)$ to a probability measure $\mu_2\in\mathcal P(\Omega)$ through
\begin{gather}\label{pWasserstein}
	W_p(\mu_1,\mu_2)\coloneqq \Big(\inf_{\gamma\in\Pi(\mu_1,\mu_2)}\int_{\Omega^2} |x-y|^p\ \dd\gamma(x,y)\Big)^{\frac{1}{p}}
\end{gather}
where the set of transport plans $\Pi(\mu_1,\mu_2)$ are all couplings $\gamma\in\mathcal P(\Omega^2)$ that have marginals $\mu_1$ and $\mu_2$ and where $p\in (1,\infty)$. $W_p$ in \eqref{pWasserstein} is called the \textit{$p$-Wasserstein distance} between $\mu_1$ and $\mu_2$. The classical optimal transport problem has been generalized in many directions. One of the generalizations, which has recently received a lot of attention, is the \textit{weak optimal transport problem}, introduced by Gozlan, Roberto, Samson, and Tetali \cite{gozlan2017}. The weak optimal transport problem for $\mu_1,\mu_2\in\mathcal P(\Omega)$ in its most general form is given by
\begin{gather*}
	\mathcal T_c(\mu_2|\mu_1)\coloneqq \inf_{\gamma\in\Pi(\mu_1,\mu_2)}\int_{\Omega}c(x,\gamma_x)\ \dd\mu_1(x)
\end{gather*}
where $\dd\gamma(x,y)=\dd \gamma_x(y)\dd\mu_1(x)$ ($\gamma_x$ is the disintegration of $\gamma$ with respect to its first marginal $\mu_1$) and where $c:\Omega\times \mathcal P(\Omega)\rightarrow\RR$ is a cost function. The importance of the weak optimal transport problem stems from its flexibility in the choice of the cost function. It generalizes the classical optimal transport problem in particular: If we pick $c(x,\gamma_x)=\int_{\Omega}|x-y|^p\ \dd \gamma_x(y)$ then $\mathcal T_c(\mu_2|\mu_1)=W_p^p(\mu_1,\mu_2)$.
The weak optimal transport problem has found applications, amongst others, in mathematical finance \cite{acciaio2021}, economics \cite{paty2022}, and the Schr{\"o}dinger problem \cite{backhoff2022}. An overview over the versatility of the weak optimal transport problem and a further literature review can be found in \cite{backhoff2022}.\\
A cost function that has received particular attention, is the \textit{barycentric cost}
\begin{gather}\label{weakbarycentriccost}
	c_p(x,\gamma_x)\coloneqq\Big|x-\int_{\Omega}y\ \dd \gamma_x(y)\Big|^p=|x-\EE[\gamma_x]|^p
\end{gather}
where $p\in (1,\infty)$. This is due to its link with $W_p$ and the theory of stochastic orderings:
\begin{definition}
	For $\mu_1,\eta_1\in\mathcal P(\Omega)$ we define the \textit{convex order} by
	\begin{gather*}
		\eta_1\leq_{\text{cx}}\mu_1\ :\Leftrightarrow \int_{\Omega}\varphi(x)\ \dd\eta_1(x)\leq\int_{\Omega}\varphi(x)\ \dd\mu_1(x)\quad \forall\varphi:\RR^d\rightarrow\RR\text{ convex.}
	\end{gather*}
\end{definition}
The intuition behind the convex order is, that if $\eta_1\leq_{\text{cx}}\mu_1$ then $\EE[\mu_1]=\EE[\eta_1]$ and $\mu_1$ is more spread out than $\eta_1$.\\
Gozlan and Juillet \cite{gozlan2020} and Alfonsi, Corbetta, and Jourdain  \cite{alfonsi2020} discovered that the \textit{weak optimal transport problem with barycentric cost}
\begin{gather*}
	\mathcal T_p(\mu_2|\mu_1)\coloneqq\mathcal T_{c_p}(\mu_2|\mu_1)=\inf_{\gamma\in\Pi(\mu_1,\mu_2)}\int_{\Omega}\Big|x-\int_{\Omega}y\ \dd \gamma_x(y)\Big|^p\ \dd\mu_1(x)
\end{gather*}
describes the $p$-Wasserstein projection of $\mu_1$ onto the set of probability measures that are in convex order with $\mu_2$. To this end, for a Borel map $T:\RR^d\rightarrow\RR^d$ and $\mu\in\mathcal P(\RR^d)$ we define the \textit{pushforward measure} $T_\#\mu\in\mathcal P(\RR^d)$ through
\begin{gather*}
	(T_\#\mu)(A)\coloneqq \mu(T^{-1}(A))
\end{gather*}
for every $A\subseteq\RR^d$ Borel.
\begin{thm}\cite[Prop 1.1]{gozlan2020},\cite[Thm 2.1]{alfonsi2020}\label{weakforwardprojthm}
	Let $p\in(1,\infty)$ and $\mu_1,\mu_2\in\mathcal P(\Omega)$. Then, there exists $\bar\eta_2\leq_{\text{cx}}\mu_2$ such that
	\begin{gather*}
		\mathcal T_p(\mu_2|\mu_1)=\inf_{\eta_2\leq_{\text{cx}}\mu_2}W_p^p(\mu_1,\eta_2)=W_p^p(\mu_1,\bar\eta_2)
	\end{gather*}
	and $\bar\eta_2$ is uniquely attained. Further, $\bar\eta_2=T_{\#}\mu_1$ where $T:\Omega\rightarrow\Omega$ is the unique map such that $(\text{id}\times T)_{\#}\mu_1$ minimizes $W_p^p(\mu_1,\bar\eta_2)$ form \eqref{pWasserstein}.
\end{thm}
Remarkably, the previous theorem holds without any assumption on $\mu_1$, especially $\mu_1$ does not need to be absolutely continuous or not charge small sets (see also \cite[Rmk 1.2]{gozlan2017}).
Theorem \ref{weakforwardprojthm} shows that the transport described by $\mathcal T_p$ happens in two steps: First, the $p$-Wasserstein optimal map $T$ (also called Brenier map) transports $\mu_1$ to the opimizer $\bar\eta_2$ in 
Theorem \ref{weakforwardprojthm}. Second, $\bar\eta_2$ and $\mu_2$ are linked through a martingale coupling, due to Strassen's theorem \cite{strassen1965}, which states that there exists a martingale coupling between two probability measures $\eta_2$ and $\mu_2$ if and only if $\eta_2\leq_{\text{cx}}\mu_2$.\\
Theorem \ref{weakforwardprojthm} frames the weak optimal transport problem as a backward projection (because $\eta_2\leq_{\text{cx}}\mu_2$), but Alfonsi, Corbetta and Jourdain \cite{alfonsi2020} showed that $\mathcal T_p$ can also be understood as a forward projection (see also \cite[Sec 8]{kim2024} for the relationship of backward and forward projections for the convex order):
\begin{thm}\cite[Thm 4.1, Cor 4.4]{alfonsi2020}\label{weakbackwardprojthm}
	Let $p\in (1,\infty)$ and $\mu_1,\mu_2\in\mathcal P(\Omega)$. Then, there exists $\bar\eta_1\geq_{\text{cx}}\mu_1$ such that
	\begin{gather*}
		\mathcal T_p(\mu_2|\mu_1)=\inf_{\mu_1\leq_{\text{cx}}\eta_1}W_p^p(\eta_1,\mu_2)=W_p^p(\bar\eta_1,\mu_2).
	\end{gather*}
	The minimizer $\bar\eta_1$ is unique if $\mu_2\in\mathcal P_{\text{ac}}(\Omega)\coloneqq\big\{\mu\in\mathcal P(\Omega):\mu<<\mathcal L^d\}$.
\end{thm}
$\mathcal T_p$ also has a dual problem, which is similar to the standard Kantorovich dual for $W_p$, but only admits concave competitors:
\begin{thm}\cite[Thm 2.11(3)]{gozlan2017}\label{weakdualitythm} Let $p\in(1,\infty)$ and $\mu_1,\mu_2\in\mathcal P(\Omega)$. Then
	\begin{gather*}
		\mathcal T_p(\mu_2|\mu_1)=\sup_{\varphi:\Omega\rightarrow\RR\text{ concave}}\int_{\Omega} \varphi^{c_p}(x_1)\ \dd\mu_1(x_1)+\int_{\Omega}\varphi(x_2)\ \dd\mu_2(x_2)
	\end{gather*}
	where
	\begin{gather*}
		\varphi^{c_p}(x_1)\coloneqq \inf_{x_2\in\Omega}\Big\{ |x_1-x_2|^p-\varphi(x_2)\Big\}.
	\end{gather*}
\end{thm}
It turns out, that many concepts from classical optimal transport can be generalized to the weak setting: Backhoff-Veraguas, Beiglb{\"o}ck and Pammer \cite{backhoff2019} found an analog of cyclical monotonicity for weak optimal transport, and simultaneously Gozlan, Le Gouic, and Samson \cite{gozlan2025} and Guo, Nilsson, and Wiesel \cite{guo2026} found a dynamic characterization for $\mathcal T_2$ analogous to the celebrated Benamou-Brenier formula \cite{benamou2000} for $W_2$.\\
In this manuscript, we continue in the spirit of generalizing from $p$-Wasserstein to $\mathcal T_p$ and study the extension of the Wasserstein barycenter problem to the weak setting. 
The Wasserstein barycenter problem is given through
\begin{gather}\label{wassersteinmidpointproblem}
	\text{WC}(\mu_1,\dots,\mu_N)\coloneqq\inf_{\nu\in\mathcal P(\Omega)}\sum_{i=1}^N\lambda_i W_p^p(\mu_i,\nu),
\end{gather}
with $\mu_1,\dots\mu_N\in\mathcal P(\Omega)$ given probability measures and where $(\lambda_i)_{i\in[N]}\in\Delta_N$ where 
\begin{gather}\label{deltaN}
	\Delta_N\coloneqq\Big\{(\lambda_i)_{i\in[N]}\in(0,1)^N:\ \sum_{i=1}^N\lambda_i=1\Big\}
\end{gather}
and we define $[N]\coloneqq\{1,\dots,N\}$.
 This problem was introduced by Agueh and Carlier \cite{agueh2011} for $p=2$ and later extended to $p\in (1,\infty)$ by Brizzi, Friesecke, and Ried \cite{brizzi2025}. The minimizer $\bar\nu$ in \eqref{wassersteinmidpointproblem} can be understood as the barycenter or centroid of the probability measures $\mu_1,\dots,\mu_N$. To avoid confusion due to the repeated use of the word barycenter in different contexts, from now on we exclusively use the word barycenter for the expected value of the disintegration of a measure as in \eqref{weakbarycentriccost}, while we call the problem \eqref{wassersteinmidpointproblem} the \textit{Wasserstein centroid problem} and we refer to its minimizer $\bar\nu$ as the Wasserstein centroid of $\mu_1,\dots,\mu_N$.\\
The Wasserstein centroid problem \eqref{wassersteinmidpointproblem} is connected to a \textit{multi-marginal optimal transport problem} (see for example \cite{ruschendorf1981,gaffke1981,kellerer1984} for early references and \cite{pass2015} for an overview of multi-marginal optimal transport). Consider $\mu_1,\dots,\mu_N\in\mathcal P(\Omega)$ and the infimal convolution cost $\bar c\in C(\Omega^N)$ given through
\begin{gather}\label{infconvcost}
	\bar c(x_1,\dots,x_N)\coloneqq\inf_{z\in\Omega}\sum_{i=1}^N\lambda_i|x_i-z|^p,
\end{gather}
 where again $(\lambda_i)_{i\in [N]}\in\Delta_N$ from \eqref{deltaN}. We define the \textit{multi-marginal optimal transport problem with infimal convolution cost} through
 \begin{gather}\label{CMM}
 	\text{CMM}(\mu_1,\dots,\mu_N)\coloneqq\inf_{\gamma\in\Pi(\mu_1,\dots,\mu_N)}\int_{\Omega^N}\bar c(x_1,\dots,x_N)\ \dd\gamma(x_1,\dots,x_N)
 \end{gather}
 where $\Pi(\mu_1,\dots,\mu_N)$ is the set of probability measures on $\Omega^N$ with marginals $\mu_1,\dots,\mu_N$ and $\bar c$ as in \eqref{infconvcost}.
Chiappori, McCann and Nesheim \cite{chiappori2010} and Carlier and Ekeland \cite{carlier2010} discovered a connection between the Wasserstein centroid and the multi-marginal optimal transport problem with infimal convolution cost. To this end, we define the map
\begin{gather}\label{barTmap}
	\bar T:
		(x_1,\dots,x_N)\in\Omega^N\mapsto \arg\inf\bar c(x_1,\dots x_N)=\arg\inf\limits_{z\in\Omega}\sum\limits_{i=1}^N\lambda_i|x_i-z|^p\in\Omega.
\end{gather}
\begin{thm}\cite[Sec 5]{chiappori2010}\cite[Prop 3]{carlier2010}\label{chiapporithm}
	Let $p\in(1,\infty)$, $\mu_1,\dots,\mu_N\in\mathcal P(\Omega)$ and $(\lambda_i)_{i\in[N]}\in\Delta_N$ from \eqref{deltaN}. Then
	\begin{gather*}
		\text{WC}(\mu_1,\dots,\mu_N)=\text{CMM}(\mu_1,\dots,\mu_N).
	\end{gather*}
	If $\bar\gamma\in\Pi(\mu_1,\dots,\mu_N)$ minimizes \eqref{CMM} then $\bar\nu\coloneqq\bar T_{\#}\bar\gamma\in\mathcal P(\Omega)$ solves \eqref{wassersteinmidpointproblem} and if $\bar\nu\in\mathcal P(\Omega)$ minimizes \eqref{wassersteinmidpointproblem} there exists $\bar\gamma\in\Pi(\mu_1,\dots,\mu_N)$ minimizing \eqref{CMM} such that $\bar\nu=\bar T_{\#}\bar\gamma$, where $\bar T$ is defined in \eqref{barTmap}.
\end{thm}
Problem $\eqref{CMM}$ also admits a dual problem, given through
\begin{gather}\label{CD}
	\text{CD}(\mu_1,\dots,\mu_N)\coloneqq \sup_{\substack{(\varphi_i)_{i\in[N]}\in C(\Omega)^N\\
	\sum_{i=1}^N\varphi_i(x_i)\leq \bar c(x_1,\dots,x_N)	
}}\sum_{i=1}^N\int_\Omega\varphi_i(x_i)\ \dd\mu_i(x_i)
\end{gather}
and then
\begin{gather*}
	\text{CMM}(\mu_1,\dots,\mu_N)=\text{CD}(\mu_1,\dots,\mu_N)
\end{gather*}
holds.\\
In this manuscript, we study the generalization of the Wasserstein centroid problem \eqref{wassersteinmidpointproblem} to the weak optimal transport problem with barycentric cost $\mathcal T_p$ for $p\in(1,\infty)$, given through
\begin{gather}\label{weakmidpointproblem}
	\text{wC}(\mu_1,\dots,\mu_N)\coloneqq\inf_{\nu\in\mathcal P(\Omega)}\sum_{i=M+1}^N\mathcal \lambda_i \mathcal T_p(\mu_i|\nu)+\sum_{i=1}^M\lambda_i\mathcal T_p(\nu|\mu_i),
\end{gather}
with $\mu_1,\dots,\mu_N\in\mathcal P(\Omega)$ and $(\lambda_i)_{i\in[N]}\in\Delta_N$ from \eqref{deltaN}. We assume that $0\leq M\leq N$ are integers. The separation between the first $M$ and the last $N-M$ indices is necessary, as $\mathcal T_p$ (unlike $W_p$) is asymmetric in its two arguments and in general $\mathcal T_p(\mu_1|\mu_2)\neq\mathcal T_p(\mu_2|\mu_1)$. We call \eqref{weakmidpointproblem} the \textit{weak centroid problem} between $\mu_1,\dots,\mu_N$ and its solution $\bar\nu\in\mathcal P(\Omega)$ a \textit{weak centroid} of $\mu_1,\dots,\mu_N$. To the best of our knowledge, \eqref{weakmidpointproblem} has not been studied in this generality. Cazelles, Tobar and Fontbona \cite{cazelles2021} have introduced \eqref{weakmidpointproblem} in the special case $M=0$ for $p=2$, but their work is mainly focussed on computing a weak centroid $\bar\nu$ algorithmically, instead of analyzing its theoretical properties.

\subsection{Notation, contributions and plan of the manuscript}
In this manuscript, we develop a theory for the weak centroid problem \eqref{weakmidpointproblem}. Section 2 of this manuscript shows existence of a minimizer $\bar\nu\in\mathcal P(\Omega)$ for \eqref{weakmidpointproblem} and comments on the cases $M=0$ and $M=N$. We especially discuss non-uniqueness of minimizers of \eqref{weakmidpointproblem} after Proposition \ref{zeroclassprop}.\\
Section 3 forms the core of our work: We define the multi-marginal optimal transport problem 
\begin{gather}\label{weakMM}
	\text{wMM}(\mu_1,\dots,\mu_N)\coloneqq\inf_{\gamma\in\Pi\big((\mu_{i\leq_{\text{cx}}})_{i=1}^M,(_{\leq_{\text{cx}}}\mu_i)_{i=M+1}^N\big)}\int_{\Omega^N}\bar c\ \dd\gamma
\end{gather}
with $\bar c$ from \eqref{infconvcost} and where for $\mu_1,\dots,\mu_N\in\mathcal P(\Omega)$ we define the set
\begin{equation}\label{Pmargset}
	\begin{gathered}
		\Pi\big((\mu_{i\leq_{\text{cx}}})_{i=1}^M,(_{\leq_{\text{cx}}}\mu_i)_{i=M+1}^N\big)\coloneqq\\
		\Big\{\gamma\in\mathcal P(\Omega^N):\mu_i\leq_{\text{cx}}\pi_{i\#}\gamma,\ i=1,\dots,M;\ \pi_{i\#}\gamma\leq_{\text{cx}}\mu_i,\ i=M+1,\dots,N\Big\}.
	\end{gathered}
\end{equation}
Here $\pi_i:\Omega^N\rightarrow\Omega$ given through $\pi_i(x_1,\dots,x_N)=x_i$ for $i\in[N]$ defines the projection on the $i$-th factor.\\
In Proposition \ref{midpoint=MMprop} we show that
\begin{gather*}
	\text{wC}(\mu_1,\dots,\mu_N)=\text{wMM}(\mu_1,\dots,\mu_N)
\end{gather*}
which is the analog of Theorem \ref{chiapporithm} for the the weak centroid problem.
Noticing, that 
\begin{gather*}
	\text{wMM}(\mu_1,\dots,\mu_N)=\inf_{\begin{subarray}{l}\mu_i\leq_{\text{cx}}\eta_i\ i=1,\dots,M\\
			\eta_i\leq_{\text{cx}}\mu_i\ i=M+1,\dots,N	\end{subarray}}\text{CMM}(\eta_1,\dots,\eta_N)
\end{gather*}
allows us to relate minimizers $\bar\nu$ of \eqref{weakmidpointproblem} to minimizers $\bar\gamma$ of \eqref{weakMM}, using Theorem \ref{chiapporithm}.\\
Multi-marginal optimal transport problems have seen some success due to their versatility in applications \cite{pass2015}, for example in modelling labour markets \cite{chiappori2010,carlier2010} or in quantum physics \cite{cotar2018}, and therefore we believe that problem \eqref{weakMM} is interesting in its own right.
To the best of our knowledge, the weak optimal transport problem has not been studied in a multi-marginal context, however, \eqref{weakMM} formulates a natural multi-marginal version for the weak optimal transport problem with barycentric cost. It encodes the inherent asymmetry of the weak optimal transport problem by requiring that each of the first $M$ marginals of $\gamma$ dominate the respective $\mu_i$ in convex order, while each of the last $N-M$ marginals of $\gamma$ are dominated by the respective $\mu_i$ in convex order.\\
Next, we find a dual problem for $\text{wMM}(\mu_1,\dots,\mu_N)$. We consider
\begin{gather}\label{weakMMdual}
	\text{wD}(\mu_1,\dots,\mu_N)\coloneqq\sup_{(\varphi_i)_{i\in[N]}\in\mathcal A_{\bar c}}\sum_{i=1}^N\int_{\Omega}\varphi_i(x_i)\ \dd\mu_i(x_i)
\end{gather}
with $\bar c$ from \eqref{infconvcost} and where we define
\begin{equation}\label{A_cset}
	\begin{gathered}
		\mathcal A_{\bar c}\coloneqq\Big\{(\varphi_i)_{i\in[N]}\in C(\Omega)^N: \sum_{i=1}^N\varphi_i(x_i)\leq \bar c(x_1,\dots,x_N),\ (x_1,\dots,x_N)\in\Omega^N;\\ \varphi_i\text{ convex for }i=1,\dots,M,\ \varphi_i\text{ concave for }i=M+1,\dots,N\Big\}.
	\end{gathered}
\end{equation}
The problem \eqref{weakMMdual} might not admit continuous, real-valued maximizers (see Remark \ref{attainmentlmarmk}), hence we consider the relaxed problem
\begin{gather}\label{weakMMdualrelaxed}
	\text{wDr}(\mu_1,\dots,\mu_N)\coloneqq\sup_{(\varphi_i)_{i\in[N]}\in\mathcal B_{\bar c}}\sum_{i=1}^M\int_{\Omega}\varphi_i\  \dd\mu_i+\sum_{i=M+1}^N\int_{K_i}\varphi_i\ \dd\mu_i
\end{gather}
with $\bar c$ from \eqref{infconvcost}, where we define $K_i\coloneqq\text{conv}(\text{supp}(\mu_i))$ for $i=M+1,\dots,N$ and the set
\begin{equation}\label{B_cset}
	\begin{gathered}
		\mathcal B_{\bar c}\coloneqq\Big\{(\varphi_i)_{i\in[N]}:\\
		\sum_{i=1}^N\varphi_i(x_i)\leq c(x_1,\dots,x_N),\ (x_1,\dots,x_N)\in\Omega^M\times K_{M+1}\times\dots\times K_N;\\ \varphi_i:\Omega\rightarrow\RR\text{ convex, continuous }\forall\ i=1,\dots,M,\\
		\varphi_i:K_i\rightarrow \RR\cup\{-\infty\}\in L^1(\mu_i)\text{ concave, upper-semicontinuous }\forall\ i=M+1,\dots,N\Big\}.
	\end{gathered}
\end{equation}
In Theorem \ref{weakMMdualthm} we show that
\begin{gather*}
	\text{wMM}(\mu_1,\dots,\mu_N)=\text{wD}(\mu_1,\dots,\mu_N)=\text{wDr}(\mu_1,\dots,\mu_N)
\end{gather*}
and Lemma \ref{dualattainmentlma} shows that $\text{wDr}(\mu_1,\dots,\mu_N)$ admits maximizers.\\
In the classical multi-marginal problem \eqref{CMM} the marginals of the couplings are fixed and the potentials in the dual problem \eqref{CD} are only continuous. For our problem \eqref{weakMM} on the other hand, we only prescribe a convex order relationship on the marginals of the couplings, which translates to requiring in the dual problem \eqref{weakMMdual} that the potentials are convex or concave, respectively.\\
Using the connection with the classical dual problem \eqref{CD}, we establish an optimality condition, relating optimizers $\bar\gamma$ of \eqref{weakMM} to optimizers $(\bar\varphi_i)_{i\in[N]}$ of \eqref{weakMMdual} in Lemma \ref{weakMMoptimality}.\\
Section 4 discusses the special case $N=2$ and $M=1$ and therefore studies the problem
\begin{gather}\label{weakmidpointproblemN2intro}
	\text{wC}(\mu_1,\mu_2)=\inf_{\nu\in\mathcal P(\Omega)}t\mathcal T_p(\mu_2|\nu)+(1-t)\mathcal T_p(\nu|\mu_1)
\end{gather}
for $t\in[0,1]$. It turns out, that for $\mu_1$ absolutely continuous, an optimizer $\bar\nu$ of \eqref{weakmidpointproblemN2intro} is given through McCann's displacement interpolant \cite{mccann1997} between $\mu_1$ and $\bar\eta_2\leq_{\text{cx}}\mu_2$ which realizes the projection in the term
\begin{gather*}
	\mathcal T_p(\mu_2|\bar\nu)=\inf_{\eta_2\leq_{\text{cx}}\mu_2}W_p^p(\eta_2,\bar\nu)
\end{gather*}
in \eqref{weakmidpointproblemN2intro}.
Using this insight, we relate \eqref{weakmidpointproblemN2intro} with $p=2$ to the dynamical version for $\mathcal T_2$ from \cite{gozlan2025} and \cite{guo2026} in Theorem \ref{weakBBrelthm}.

\section{Basics about the weak centroid problem}
We fix a compact and convex set $\Omega\subset\RR^d$ and integers $0\leq M\leq N$ for $N>0$. We expect all of our results to also hold if $\Omega=\RR^d$ for $\mu_1,\dots,\mu_N\in\mathcal P_p(\RR^d)\coloneqq\big\{\mu\in\mathcal P(\RR^d):\int_{\RR^d}|x|^p\ \dd\mu(x)<\infty\}$, but for simplicity we will only work on a compact set $\Omega$. During this manuscript, $C$ wil denote a generic constant in $\RR$.
\begin{lma}
	Let $p\in (1,\infty)$,  $\mu_1,\dots,\mu_N\in\mathcal P(\Omega)$ and $(\lambda_i)_{i\in[N]}\in\Delta_N$ from \eqref{deltaN}. Then \eqref{weakmidpointproblem} admits a minimizer $\bar\nu\in\mathcal P(\Omega)$.
\end{lma}
We adapt the proof from \cite[Prop 1]{cazelles2021}:
\begin{proof}
	Let $(\nu_n)_{n\in\NN}\subset\mathcal P(\Omega)$ be a minimizing sequence of 
	\begin{gather*}
		\nu\mapsto F(\nu)\coloneqq \sum_{i=M+1}^N\mathcal \lambda_i \mathcal T_p(\mu_i|\nu)+\sum_{i=1}^M\lambda_i\mathcal T_p(\nu|\mu_i).
	\end{gather*}
	By Prokhorov's theorem, there exists a subsequence of $(\nu_n)_{n\in\NN}$ (still denoted by $(\nu_n)_{n\in\NN}$) converging weakly to some $\bar\nu\in\mathcal P(\Omega)$. \cite[Thm 1.5]{backhoff2020} shows that $\nu_n\rightarrow\bar\nu$ weakly implies $\mathcal T_p(\mu_i|\nu_n)\rightarrow\mathcal T_p(\mu_i|\bar\nu)$ and $\mathcal T_p(\nu_n|\mu_i)\rightarrow\mathcal T_p(\bar\nu|\mu_i)$ and hence
	\begin{gather*}
		\min_{\nu\in\mathcal P(\Omega)}F(\nu)=\lim_{n\rightarrow\infty}F(\nu_n)=F(\bar\nu).
	\end{gather*}
\end{proof}
For the rest of the manuscript, $\bar\nu\in\mathcal P(\Omega)$ denotes a minimizer of \eqref{weakmidpointproblem}.
\begin{prop}\label{zeroclassprop}
	Let $p\in(1,\infty)$, $\mu_1,\dots,\mu_N\in\mathcal P(\Omega)$ and 
	$(\lambda_i)_{i\in[N]}\in\Delta_N$ from \eqref{deltaN}. Then $\text{wC}(\mu_1,\dots,\mu_N)=0$ if and only if there exists $\bar\nu\in\mathcal P(\Omega)$ such that $\mu_1,\dots,\mu_M\leq_{\text{cx}}\bar\nu\leq_{\text{cx}}\mu_{M+1},\dots,\mu_{N}$. In this case $\bar\nu$ minimizes \eqref{weakmidpointproblem}.
\end{prop}
\begin{proof}
	$\text{wC}(\mu_1,\dots,\mu_N)=0$ if and only if there exists $\bar\nu\in\mathcal P(\Omega)$ such that $\mathcal T_p(\mu_i|\bar\nu)=0$ for $i=M+1,\dots,N$ and $\mathcal T_p(\bar\nu|\mu_i)=0$ for $i=1,\dots,M$ which holds if and only if $\mu_1,\dots,\mu_M\leq_{\text{cx}}\bar\nu\leq_{\text{cx}}\mu_{M+1},\dots,\mu_{N}$.
\end{proof}
Proposition \ref{zeroclassprop} shows, that minimizers $\bar\nu$ of \eqref{weakmidpointproblem} are non-unique: Assume that
\begin{gather*}
	\mu_1\leq_{\text{cx}}\dots\leq_{\text{cx}}\mu_M\leq_{\text{cx}}\mu_{M+1}\leq_{\text{cx}}\dots\leq_{\text{cx}}\mu_{N}
\end{gather*}then every $\bar\nu\in\mathcal P(\Omega)$ that fulfills $\mu_M\leq_{\text{cx}}\bar\nu\leq_{\text{cx}}\mu_{M+1}$ minimizes \eqref{weakmidpointproblem} with\\ $\text{wB}(\mu_1,\dots,\mu_N)=0$.

\subsection{The one-sided weak centroid problems}
In this subsection, we discuss the special cases of the one-sided version of the weak centroid problem \eqref{weakmidpointproblem}, corresponding to $M=0$ and $M=N$.
We first discuss the case $M=0$, corresponding to $\eqref{weakmidpointproblem}$ taking the simplified form
\begin{gather}\label{M0}
	\inf_{\nu\in\mathcal P(\Omega)}\sum_{i=1}^N \lambda_i \mathcal T_p(\mu_i|\nu).
\end{gather}
This is precisely the setting discussed by Cazelles, Tobar and Fontbona \cite{cazelles2021} for $p=2$. They show in \cite[Lma 1]{cazelles2021}, that if $\bar\nu$ minimizes \eqref{M0}, then every $\nu\leq_{\text{cx}}\bar\nu$ minimizes $\eqref{M0}$ - especially $\delta_{\EE[\bar\nu]}$ is a minimizer. Combining this with Proposition \ref{zeroclassprop} we get the following:
\begin{prop}\label{M0zeroclassprop}
	Let $p\in(1,\infty)$, $\mu_1,\dots,\mu_N\in\mathcal P(\Omega)$, $(\lambda_i)_{i\in[N]}\in\Delta_N$ from \eqref{deltaN} and $M=0$. Then $\text{wC}(\mu_1,\dots,\mu_N)=0$ if and only if $\EE[\mu_1]=\dots=\EE[\mu_N]$.
\end{prop}
For the case $M=N$, \eqref{weakmidpointproblem} takes the form
\begin{gather}\label{MN}
	\inf_{\nu\in\mathcal P(\Omega)}\sum_{i=1}^N\lambda_i\mathcal T_p(\nu|\mu_i).
\end{gather}
If $\Omega$ was equal to $\RR^d$, then for $\mu_1,\dots,\mu_N\in\mathcal P_1(\RR^d)$ with $m\coloneqq\EE[\mu_1]=\dots=\EE[\mu_N]$ we could construct the random variable $Z\coloneqq\sum_{i=1}^NX_i-(N-1)m$ where $X_i\sim\mu_i$ are independent random variables. Then $\EE[Z|X_i]=X_i$ for every $i\in[N]$ which by Strassen's theorem implies that $\text{law}(Z)\geq_{\text{cx}}\mu_i$ for every $i\in[N]$. Therefore, we would get an analog of Proposition \ref{M0zeroclassprop} stating that \eqref{MN} is zero if and only if $\EE[\mu_1]=\dots=\EE[\mu_N]$.\\
However, as $\Omega$ is compact, there might not always be a measure, that is supported on $\Omega$ and higher in convex order than two given measures. This can be seen from the example $d=2$ with $\Omega$ the closed unit disc and probability measures $\mu_1=\frac{1}{2}(\delta_{e_1}+\delta_{-e_1})$ and $\mu_2=\frac{1}{2}(\delta_{e_2}+\delta_{-e_2})$, which do not admit a common higher element in convex order that is supported on $\Omega$.
\section{Multi-marginal formulation and duality}
\begin{lma}
	Let $p\in(1,\infty)$, $\mu_1,\dots,\mu_N\in\mathcal P(\Omega)$ and $(\lambda_i)_{i\in[N]}\in\Delta_N$ from \eqref{deltaN}. Then
	\eqref{weakMM} admits a minimizer $\bar\gamma\in \Pi\big((\mu_{i\leq_{\text{cx}}})_{i=1}^M,(_{\leq_{\text{cx}}}\mu_i)_{i=M+1}^N\big)$ from \eqref{Pmargset}.
\end{lma}
\begin{proof}
	The set $\Pi\big((\mu_{i\leq_{\text{cx}}})_{i=1}^M,(_{\leq_{\text{cx}}}\mu_i)_{i=M+1}^N\big)$ from \eqref{Pmargset} is non-empty as it contains the product measure $\mu_1\otimes\dots\otimes\mu_N$. Let $(\gamma_n)_{n\in\NN}\subset\Pi\big((\mu_{i\leq_{\text{cx}}})_{i=1}^M,(_{\leq_{\text{cx}}}\mu_i)_{i=M+1}^N\big)$ be a minimizing sequence for 
	\begin{gather*}
		F(\gamma)\coloneqq\int_{\Omega^N}\bar c(x_1,\dots,x_N)\ \dd\gamma(x_1,\dots,x_N).
	\end{gather*}
	As $\Omega^N$ is compact, $(\gamma_n)_{n\in\NN}$ is tight and admits a subsequence (denoted again by $(\gamma_n)_{n\in\NN}$) that converges weakly to some $\bar\gamma\in\mathcal P(\Omega^N)$.\\
	Let $\varphi:\RR^d\rightarrow\RR$ be a convex function, then $\varphi$ is continuous \cite[Cor 10.1.1]{rockafellar1997}. Then for $i=1,\dots,M$ and all $n\in\NN$
	\begin{gather*}
		\int_\Omega\varphi(x_i)\ \dd\mu_i(x_i)\leq\int_\Omega\varphi(x_i)\ \dd(\pi_{i\#}\gamma_n)(x_i)
	\end{gather*}
	and hence
	\begin{gather*}
		\int_\Omega\varphi(x_i)\ \dd\mu_i(x_i)\leq\lim_{n\in\NN}\int_\Omega\varphi(x_i)\ \dd(\pi_{i\#}\gamma_n)(x_i)=\int_\Omega\varphi(x_i)\ \dd(\pi_{i\#}\bar\gamma)(x_i)
	\end{gather*}
	(a similar argument applies for $i=M+1,\dots,N$) and therefore\\ $\bar\gamma\in\Pi\big((\mu_{i\leq_{\text{cx}}})_{i=1}^M,(_{\leq_{\text{cx}}}\mu_i)_{i=M+1}^N\big)$.
	Clearly $F$ is continuous with respect to the weak convergence and hence
	\begin{gather*}
		\inf_{\gamma\in\Pi\big((\mu_{i\leq_{\text{cx}}})_{i=1}^M,(_{\leq_{\text{cx}}}\mu_i)_{i=M+1}^N\big)}F(\gamma)=\lim_{n\in\NN}F(\gamma_n)=F(\bar\gamma).
	\end{gather*}
\end{proof}
From now on, we denote by $\bar\gamma\in\Pi\big((\mu_{i\leq_{\text{cx}}})_{i=1}^M,(_{\leq_{\text{cx}}}\mu_i)_{i=M+1}^N\big)$ a minimizer of \eqref{weakMM}.
We now show equality between the weak centroid problem \eqref{weakmidpointproblem} and the multi-marginal problem \eqref{weakMM} and the relationship between their optimizers.

\begin{prop}\label{midpoint=MMprop}
	Let $p\in(1,\infty)$, $\mu_1,\dots,\mu_N\in\mathcal P(\Omega)$ and $(\lambda_i)_{i\in[N]}\in\Delta_N$ from \eqref{deltaN}. Then the weak centroid problem \eqref{weakmidpointproblem} equals the multi-marginal problem \eqref{weakMM},
	\begin{gather*}
			\text{wC}(\mu_1,\dots,\mu_N)=\text{wMM}(\mu_1,\dots,\mu_N).
	\end{gather*}
	If $\bar\gamma\in\Pi\big((\mu_{i\leq_{\text{cx}}})_{i=1}^M,(_{\leq_{\text{cx}}}\mu_i)_{i=M+1}^N\big)$ from \eqref{Pmargset} minimizes \eqref{weakMM}, then $\bar\nu\coloneqq\bar T_{\#}\bar\gamma\in\mathcal P(\Omega)$ minimizes \eqref{weakmidpointproblem}, where $\bar T$ is the map from \eqref{barTmap}. If $\bar\nu\in\mathcal P(\Omega)$ minimizes \eqref{weakmidpointproblem}, then there exists $\bar\gamma\in\Pi\big((\mu_{i\leq_{\text{cx}}})_{i=1}^M,(_{\leq_{\text{cx}}}\mu_i)_{i=M+1}^N\big)$ that minimizes \eqref{weakMM} such that $\bar\nu=\bar T_{\#}\bar\gamma$.
\end{prop}
\begin{proof}
	We get by Theorem \ref{weakforwardprojthm} and Theorem \ref{weakbackwardprojthm}
	\begin{align*}
		\text{wC}(\mu_1,\dots,\mu_N)
		&=\inf_{\nu\in\mathcal P(\Omega)}\sum_{i=M+1}^N\mathcal \lambda_i \inf_{\eta_i\leq_{\text{cx}}\mu_i}W_p^p(\eta_i,\nu)+\sum_{i=1}^M\lambda_i\inf_{\mu_i\leq_{\text{cx}}\eta_i}W_p^p(\nu,\eta_i)\\
		&=\inf_{\substack{\mu_i\leq_{\text{cx}}\eta_i\\ i=1,\dots,M}}\inf_{\substack{\eta_i\leq_{\text{cx}}\mu_i\\ i=M+1,\dots,N}}\text{CMM}(\eta_1,\dots,\eta_N)\\
		&=\text{wMM}(\mu_1,\dots,\mu_N),
	\end{align*}
	where the second equality and the relationship between the minimizers is due to Theorem \ref{chiapporithm}.
\end{proof}
Next, we discuss the dual problem for $\text{wMM}(\mu_1,\dots,\mu_N)$:

\begin{thm}\label{weakMMdualthm}
	Let $p\in (1,\infty)$,  $\mu_1,\dots,\mu_N\in\mathcal P(\Omega)$ and $(\lambda_i)_{i\in[N]}\in\Delta_N$ from \eqref{deltaN}. Then we get equality of \eqref{weakmidpointproblem}, \eqref{weakMMdual}, and \eqref{weakMMdualrelaxed}, i.e.
	\begin{gather*}
		\text{wMM}(\mu_1,\dots,\mu_N)=\text{wD}(\mu_1,\dots,\mu_N)=\text{wDr}(\mu_1,\dots,\mu_N).
	\end{gather*}
\end{thm}
The proof of this equality follows the standard proof of duality in optimal transport theory (see for example \cite[Thm 1.3]{villani2003}), using the Fenchel-Rockafellar duality theorem.
\begin{proof}
	Define $F:C(\Omega)^N\rightarrow\RR\cup\{+\infty\}$ through
	\begin{gather*}
		F((\varphi_i)_{i\in[N]})\coloneqq\\
		\begin{cases}
			-\sum_{i=1}^n\int_\Omega\varphi_i\ \dd\mu_i&\varphi_i\text{ convex for }i=1,\dots,M\ \varphi_i\text{ concave for }i=M+1,\dots,N\\
			+\infty&\text{else,}
		\end{cases}
	\end{gather*}
	define $G:C(\Omega^N)\rightarrow\RR\cup\{+\infty\}$ through
	\begin{gather*}
		G(f)\coloneqq
		\begin{cases}
			0&f\leq \bar c\\
			+\infty&\text{else}
		\end{cases}
	\end{gather*}
	and further define the operator $A:C(\Omega)^N\rightarrow C(\Omega^N)$ by
	\begin{gather*}
		A((\varphi_i)_{i\in[N]})(x_1,\dots,x_n)=\sum_{i=1}^N\varphi_i(x_i).
	\end{gather*}
	Note that for $(\tilde\varphi_i)_{i\in[N]}\coloneqq(C,0,\dots,0)$ for $C<0$ we get
	$F((\tilde\varphi_i)_{i\in[N]})<+\infty$ and $G=0$ in a neighbourhood of $(\tilde\varphi_i)_{i\in[N]}$, hence $G$ is continuous at $(\tilde\varphi_i)_{i\in[N]}$ and hence the qualifying condition holds.
	We therefore get by Fenchel-Rockafellar duality (see for example \cite[Thm 4.1, Rmk 4.2]{ekeland1976}) that
	\begin{align*}
		-\text{wD}(\mu_1,\dots,\mu_N)&=\inf_{(\varphi_i)_{i\in[N]}\in C(\Omega)^N} F((\varphi_i)_{i\in[N]})+ G(A((\varphi_i)_{i\in[N]}))\\
		&=\sup_{\gamma\in\mathcal M(\Omega^N)}-F^*(-A^*(\gamma))-G^*(\gamma).
	\end{align*}
	We compute
	\begin{align*}
		G^*(\gamma)&=\sup_{f\in C(\Omega^N)}\int_{\Omega^N}f\ \dd\gamma-G(f)\\
		&=\sup_{\substack{f\in C(\Omega^N)\\ f\leq \bar c}}\int_{\Omega^N}f\ \dd\gamma.
	\end{align*}
	We get $A^*(\gamma)=(\pi_{i\#}\gamma)_{i\in[N]}$ and hence
	\begin{align*}
		F^*(-A^*\gamma)&=\sup_{(\varphi_i)_{i\in[N]}}\sum_{i=1}^N\int_\Omega-\varphi_i\ \dd(\pi_{i\#}\gamma)-F((\varphi_i)_{i\in[N]})\\
		&=\sum_{i=1}^M\sup_{\varphi_i\in C(\Omega)\text{ convex}}\int_\Omega\varphi_i\ \dd(\mu_i-\pi_{i\#}\gamma)\\
		&+\sum_{i=M+1}^N\sup_{\varphi_i\in C(\Omega)\text{ concave}}\int_\Omega\varphi_i\ \dd(\mu_i-\pi_{i\#}\gamma)
	\end{align*}
	and therefore
	\begin{align*}
		&\sup_{\gamma\in\mathcal M(\Omega^N)}-F^*(-A^*(\gamma))-G^*(\gamma)\\
		&=-\inf_{\gamma\in\mathcal M(\Omega^N)}\sum_{i=1}^M\sup_{\varphi_i\in C(\Omega)\text{ convex}}\int_\Omega\varphi_i\ \dd(\mu_i-\pi_{i\#}\gamma)\\
		&+\sum_{i=M+1}^N\sup_{\varphi_i\in C(\Omega)\text{ concave}}\int_\Omega\varphi_i\ \dd(\mu_i-\pi_{i\#}\gamma)+\sup_{\substack{f\in C(\Omega^N)\\ f\leq \bar c}}\int_{\Omega^N}f\ \dd\gamma.
	\end{align*}
	Let $\gamma\notin\mathcal M_+(\Omega^N)$ then we can pick $f_0\leq 0$ with $\int_{\Omega^N}f_0\ \dd\gamma> 0$ to achieve
	\begin{gather*}
		\sup\limits_{\substack{f\in C(\Omega^N)\\ f\leq \bar c}}\int_{\Omega^N}f\ \dd\gamma\rightarrow+\infty.
	\end{gather*}
	Hence, the minimization forces $\gamma\in\mathcal M_+(\Omega^N)$ and we pick $f=\bar c$. Similarly, finiteness of the minimization requires $\pi_{i\#}\gamma\geq_{\text{cx}}\mu_i$ for $i=1,\dots,M$ and $\pi_{i\#}\gamma\leq_{\text{cx}}\mu_i$ for $i=M+1,\dots,N$. Therefore, we get
	\begin{align*}
		-\text{wD}(\mu_1,\dots,\mu_N)&=\inf_{(\varphi_i)_{i\in[N]}\in C(\Omega)^N} F((\varphi_i)_{i\in[N]})+ G(A((\varphi_i)_{i\in[N]}))\\
		&=\sup_{\gamma\in\mathcal M(\Omega^N)}-F^*(-A^*(\gamma))-G^*(\gamma)\\
		&=-\text{wMM}(\mu_1,\dots,\mu_N).
	\end{align*}
	We also clearly have
	\begin{gather*}
		\text{wD}(\mu_1,\dots,\mu_N)\leq\text{wDr}(\mu_1,\dots,\mu_N).
	\end{gather*}
	Integrating the inequality
	\begin{gather*}
		\sum_{i=1}^N\varphi_i(x_i)\leq \bar c(x_1,\dots,x_N)
	\end{gather*}
	with respect to $\gamma\in\Pi\big((\mu_{i\leq_{\text{cx}}})_{i=1}^M,(_{\leq_{\text{cx}}}\mu_i)_{i=M+1}^N\big)$ finally gives 
	\begin{gather*}
		\text{wDr}(\mu_1,\dots,\mu_N)\leq\text{wMM}(\mu_1,\dots,\mu_N)
	\end{gather*}
	which concludes the proof.
\end{proof}
\begin{lma}\label{dualattainmentlma}
	Let $p\in (1,\infty)$, $\mu_1,\dots,\mu_M\in\mathcal P(\Omega)$, $\mu_{M+1},\dots,\mu_N\in\mathcal P_{\text{ac}}(\Omega)$ and $(\lambda_i)_{i\in[N]}\in\Delta_N$ from \eqref{deltaN}, then the problem \eqref{weakMMdualrelaxed}
	admits maximizers $(\bar\varphi_i)_{i\in [N]}\in\mathcal B_{\bar c}$ from \eqref{B_cset}.
\end{lma}
The proof of Lemma \ref{dualattainmentlma} can be found in the appendix, but the following remark sketches its idea:
\begin{rmk}\label{attainmentlmarmk}
	The standard strategy to show attainment of the dual problem in optimal transport (see for example \cite[Prop 1.11]{santambrogio2015}) would be to pick a maximizing sequence $(\varphi_1^n,\dots,\varphi_N^n)_{n\in\NN}\subset\mathcal A_{\bar c}$ for \eqref{weakMMdual} (which is automatically a maximizing sequence for \eqref{weakMMdualrelaxed}) and to write for each $j=1,\dots,N$
	\begin{gather}\label{feasibilitycond}
		\varphi_j^n(x_j)\leq\bar c(x_1,\dots,x_N)-\sum_{i\neq j}\varphi_i^n(x_i)
	\end{gather}
	and, in order to improve the value of
	\begin{gather*}
		(\varphi_1,\dots,\varphi_N)\mapsto\sum_{i=1}^N\int_\Omega\varphi_i(x_i)\ \dd\mu_i(x_i),
	\end{gather*}
	while preserving feasibility, to replace
	\begin{gather*}
		\varphi_j^n(x_j)\rightarrow\hat\varphi_j^n(x_j)\coloneqq \inf_{(x_i)_{i\neq j}}\bar c(x_1,\dots,x_N)-\sum_{i\neq j}\varphi_i^n(x_i).
	\end{gather*}
	However, it is not obvious that $\hat\varphi_j^n(x_j)$ is still convex or concave. Hence, for $j=1,\dots,M$ we replace $\varphi_j^n(x_j)\rightarrow\hat\varphi_j^{n**}(x_j)$, using the property that the biconjugate is the largest convex minorant of a function. Normalizing additionally
	\begin{gather*}
		\varphi_1^n(\EE[\mu_1])=\dots=\varphi_{N-1}^n(\EE[\mu_{N-1}])=0
	\end{gather*}
	allows us to derive universal global bounds, independent of $x_j$ and $n$ such that $-C\leq \varphi_j^n(x_j)\leq C$ for each $j=1,\dots,M$.\\
	For the concave potentials $\varphi_{M+1}^n,\dots,\varphi_N^n$ there exists no canonical replacement similar to the biconjugate. While we are still able to derive global upper bounds for the concave potentials using \eqref{feasibilitycond}, the only information that we have available to derive a lower bound is that for each $j=M+1,\dots,N$ we get the integral bound
	\begin{gather*}
		-C\leq \int_{K_j}\varphi_j^n(x_j)\ \dd\mu_j(x_j)\leq C
	\end{gather*}
	with $C$ independent of $n$. This integral bound allows us to derive a uniform local lower bound on a compact set in the (relative) interior of $K_j=\text{conv}(\text{supp}(\mu_j))$, using Theorem \ref{csiszarthm} due to Csiszár and Matú\v{s} \cite{csiszar2001}.\\
	Having established uniform boundedness for the convex potentials on the whole of $\Omega$ and uniform boundedness of the concave potentials only on a compact subset of $K_j$ allows us to apply Arzela-Ascoli and find a pointwise limit $(\bar\varphi_1,\dots,\bar\varphi_N)$ of $(\varphi_1^n,\dots,\varphi_N^n)_{n\in\NN}$. For $j=M+1,\dots,N$ the limit $\bar\varphi_j$ is only defined in $\text{int}(K_j)$, but we extend it in an upper-semicontinuous way to the boundary of $K_j$. Due to the absolute continuity of $\mu_j$, this does not change the value of $\int\bar\varphi_j\ \dd\mu_j$ and as such produces admissible maximizers for \eqref{weakMMdualrelaxed} --- however, it destroys the finiteness and continuity required of maximizers for \eqref{weakMMdual}.
\end{rmk}
We also get an optimality condition relating the solutions of $\text{wMM}(\mu_1,\dots,\mu_N)$ and $\text{wD}(\mu_1,\dots,\mu_n)$ by using the classical optimality condition relating solutions of \eqref{CMM} and \eqref{CD}:
\begin{lma}\label{weakMMoptimality}
	Let $p\in(1,\infty)$, $\mu_1,\dots,\mu_N\in\mathcal P(\Omega)$ and $(\lambda_i)_{i\in [N]}\in\Delta_N$ from \eqref{deltaN}.
	Let $(\varphi_{i})_{i\in[N]}\in\mathcal B_{\bar c}$ from \eqref{B_cset} and let $\gamma\in \Pi((\mu_{i\leq_{\text{cx}}})_{i=1}^M,(_{\leq_{\text{cx}}}\mu_i)_{i=M+1}^N)$ from \eqref{Pmargset}. TFAE:
	\begin{enumerate}
		\item $\gamma$ solves $\text{wMM}(\mu_1,\dots,\mu_N)$ from \eqref{weakMM} and $(\varphi_i)_{i\in[N]}$ solve $\text{wDr}(\mu_1,\dots,\mu_N)$ from \eqref{weakMMdualrelaxed}
		\item for $\gamma$-a.e. $(x_1,\dots,x_N)\in\Omega^N$,
		\begin{gather}\label{MMoptimality}
			\sum_{i=1}^N\varphi_i(x_i)=\bar c(x_1,\dots,x_N)
		\end{gather}
		holds for $\bar c$ from \eqref{infconvcost} and for all $i\in[N]$ we get
		\begin{gather}\label{massbalanceoptimality}
			\int_\Omega\varphi_i(x_i)\ \dd(\pi_{i\#}\gamma)(x_i)=\int_\Omega\varphi_i(x_i)\ \dd\mu_i(x_i).
		\end{gather}
	\end{enumerate}
\end{lma}
\begin{rmk}
	Note that in \eqref{massbalanceoptimality} the inequality
	\begin{gather*}
		\int_\Omega\varphi_i(x_i)\ \dd(\pi_{i\#}\gamma)(x_i)\geq\int_\Omega\varphi_i(x_i)\ \dd\mu_i(x_i)
	\end{gather*}
	always holds.
	The condition \eqref{massbalanceoptimality} already appeared in \cite[Lma 2]{gozlan2025} in the context of the dual problem for (the two-marginal problem) $\mathcal T_2(\mu_2|\mu_1)=W_2^2(\mu_1,\bar\eta_2)$ and is there subsequently exploited \cite[Sec 3]{gozlan2025} to construct a martingale-invariant paving between $\bar\eta_2$ and $\mu_2$.
\end{rmk}
\begin{proof}
	1. $\Rightarrow$ 2.:\\
	Let $\gamma$ be optimal for $\text{wMM}(\mu_1,\dots,\mu_N)$ and $(\varphi_i)_{i\in[N]}$ solve $\text{wDr}(\mu_1,\dots,\mu_N)$. Notice that $\gamma$ is clearly a minimizer in $\text{CMM}(\eta_1,\dots,\eta_N)$ where $\eta_i\coloneqq \pi_{i\#}\gamma$ for $i\in[N]$.
	Hence we get
	\begin{align*}
		&\sum_{i=1}^N\int_\Omega\varphi_i(x_i)\ \dd\mu_i(x_i)\\
		&\leq\sum_{i=1}^N\int_\Omega\varphi_i(x_i)\ \dd\eta_i(x_i)&&(\mu_i\leq_{\text{cx}}\eta_i\ i=1,\dots,M;\ \eta_i\leq_{\text{cx}}\mu_i\ i=M+1,\dots,N)\\
		&\leq\text{CD}(\eta_1,\dots,\eta_N)\\
		&=\text{CMM}(\eta_1,\dots,\eta_N)\\
		&=\text{wMM}(\mu_1,\dots,\mu_N)\\
		&=\sum_{i=1}^N\int_\Omega\varphi_i(x_i)\ \dd\mu_i(x_i).
	\end{align*}
	Therefore, equality holds, we get \eqref{massbalanceoptimality}
	and $(\varphi_i)_{i\in[N]}$ solve $\text{CD}(\eta_1,\dots,\eta_N)$. Hence, from the classical optimality relation for $\text{CD}(\eta_1,\dots,\eta_N)$ and $\text{CMM}(\eta_1,\dots,\eta_N)$ (see for example \cite[Prop 2.3]{kim2014}) the statement
	\eqref{MMoptimality}
	holds for $\gamma$-a.e. $(x_1,\dots,x_N)\in\Omega^N$.\\
	2. $\Rightarrow$ 1.:\\
	Let \eqref{MMoptimality} and \eqref{massbalanceoptimality} hold, then we get
	\begin{align*}
		&\sum_{i=1}^N\int_\Omega\varphi_i(x_i)\ \dd\mu_i(x_i)\\
		&=\sum_{i=1}^N\int_\Omega\varphi_i(x_i)\ \dd(\pi_{i\#}\gamma)(x_i)&&(\text{due to \eqref{massbalanceoptimality}})\\
		&=\int_{\Omega^N}\bar c(x_1,\dots,x_N)\ \dd\gamma(x_1,\dots,x_N)&&(\text{due to \eqref{MMoptimality}})\\
		&\geq\text{wMM}(\mu_1,,\dots,\mu_N)\\
		&=\text{wDr}(\mu_1,\dots,\mu_N)\\
		&\geq\sum_{i=1}^N\int_\Omega\varphi_i(x_i)\ \dd\mu_i(x_i)
	\end{align*}
	and hence equality holds, $(\varphi_i)_{i\in[N]}$ maximize $\text{wDr}(\mu_1,\dots,\mu_N)$ and $\gamma$ minimizes $\text{wMM}(\mu_1,\dots,\mu_N)$.
\end{proof}
\section{The weak interpolation problem}
Here, we consider the special case $N=2$ with $M=1$.
In this setting, we can give a more explicit description of a minimizer $\bar\nu$. Therefore, for $\mu_1,\mu_2\in\mathcal P(\Omega)$ and $t\in[0,1]$ we study the problem
\begin{gather}\label{weakmidpointproblemN2}
	\text{wC}(\mu_1,\mu_2)=\inf_{\nu\in\mathcal P(\Omega)}t\mathcal T_p(\mu_2|\nu)+(1-t)\mathcal T_p(\nu|\mu_1)
\end{gather}
together with its dual problem
\begin{gather}\label{weakdualproblemN2}
	\text{wD}(\mu_1,\mu_2)=\sup_{(\varphi_1,\varphi_2)\in\mathcal A_{\bar c_t}}\int_\Omega\varphi_1(x_1)\ \dd\mu_1(x_1)+\int_\Omega\varphi_2(x_2)\ \dd\mu_2(x_2)
\end{gather}
where $\mathcal A_{c_t}$ from \eqref{A_cset} and
\begin{gather*}
	\bar c_t(x_1,x_2)\coloneqq \inf_{z\in\Omega}t|x_2-z|^p+(1-t)|z-x_1|^p.
\end{gather*}
We recall that $\varphi_1$ is convex and $\varphi_2$ is concave and fulfill
\begin{gather*}
	\varphi_1(x_1)+\varphi_2(x_2)\leq\bar c_t(x_1,x_2).
\end{gather*}
Hence, using standard ideas \cite[p 11]{santambrogio2015}, we would like to pick
\begin{gather*}
	\varphi_1(x_1)=\inf_{x_2\in\Omega}\bar c_t(x_1,x_2)-\varphi_2(x_2)
\end{gather*}
to make it as large as possible and we note that the function
$$x_1\mapsto\inf_{x_2\in\Omega}\bar c_t(x_1,x_2)-\varphi_2(x_2)$$ is automatically convex as an infimal convolution of two convex fuctions.\\
Hence, we get
\begin{align*}
	&\text{wD}(\mu_1,\mu_2)\\
	&=\sup_{\substack{\varphi_1\in C(\Omega)\text{ convex},\
			\varphi_2\in C(\Omega)\text{ concave}\\
			\varphi_1(x_1)+\varphi_2(x_2)\leq\bar c_t(x_1,x_2)}}\int_\Omega\varphi_1(x_1)\ \dd\mu_1(x_1)+\int_\Omega\varphi_2(x_2)\ \dd\mu_2(x_2)\\
	&=\sup_{\varphi_2\in C(\Omega)\text{ concave}}\int_\Omega\varphi_2^{\bar c_t}(x_1)\ \dd\mu_1(x_1)+\int_\Omega\varphi_2(x_2)\ \dd\mu_2(x_2)\\
	&=\sup_{\substack{\varphi_1\in C(\Omega),\\
			\varphi_2\in C(\Omega)\text{ concave},\\
			\varphi_1(x_1)+\varphi_2(x_2)\leq\bar c_t(x_1,x_2)}}\int_\Omega\varphi_1(x_1)\ \dd\mu_1(x_1)+\int_\Omega\varphi_2(x_2)\ \dd\mu_2(x_2).
\end{align*}
Having $\varphi_1\in C(\Omega)$ unbounded (i.e not a-priori convex) forces $\pi_{1\#}\gamma=\mu_1$ in the proof of Theorem \ref{weakMMdualthm} and hence we get
\begin{thm}\label{weakdualityN2thm}
	Let $p\in(1,\infty)$, $\mu_1,\mu_2\in P(\Omega)$ and $t\in[0,1]$. Then
	\begin{align*}
		\text{wC}(\mu_1,\mu_2)&=\sup_{\substack{\varphi_1\in C(\Omega),\\ \varphi_2\in C(\Omega)\text{ concave},\\
				\varphi_1(x_1)+\varphi_2(x_2)\leq\bar c_t(x_1,x_2)}}\int_\Omega\varphi_1(x_1)\ \dd\mu_1(x_1)+\int_\Omega\varphi_2(x_2)\ \dd\mu_2(x_2)\\
		&=\inf_{\eta_2\leq_{\text{cx}}\mu_2}\text{CMM}(\mu_1,\eta_2)\\
		&=\inf_{\eta_2\leq_{\text{cx}}\mu_2}\text{WC}(\mu_1,\eta_2).
	\end{align*}
	Further, $\bar\eta_2=\arg\inf\limits_{\eta_2\leq_{\text{cx}}\mu_2}\text{CMM}(\mu_1,\eta_2)$ is attained uniquely.
\end{thm}
\begin{proof}
	The equality of $\text{wC}(\mu_1,\mu_2)$ with the (relaxed) dual problem follows from the argument above. The connection with the two-marginal transport problem with infimal convolution cost comes from modifying the proof of Theorem \ref{weakMMdualthm} - choosing $\varphi_1$ unbounded corresponds to forcing $\pi_{1\#}\gamma=\mu_1$ in $\text{wMM}(\mu_1,\mu_2)$ which then becomes $\inf_{\eta_2\leq_{\text{cx}}\mu_2}\text{CMM}(\mu_1,\eta_2)$. The equality $\text{CMM}(\mu_1,\eta_2)=\text{WC}(\mu_1,\eta_2)$ is from Theorem \ref{chiapporithm}. From the equality $\text{CMM}(\mu_1,\eta_2)=C(p,t)W_p^p(\mu_1,\eta_2)$ in Lemma \ref{cMMW_plma} in the Appendix we get uniqueness of $\bar\eta_2$.
\end{proof}
From now on we denote by $\bar\eta_2$ the unique minimizer of $\inf\limits_{\eta_2\leq_{\text{cx}}\mu_2}\text{CMM}(\mu_1,\eta_2)$.
Having established the above theorem, we deduce the following chain of equalities:
\begin{align*}
	&\inf_{\nu\in\mathcal P(\Omega)}t\mathcal T_p(\mu_2|\nu)+(1-t)\mathcal T_p(\nu|\mu_1)\\
	&=\text{wC}(\mu_1,\mu_2)\\
	&=\inf_{\eta_2\leq_{\text{cx}}\mu_2}\text{CMM}(\mu_1,\eta_2)\\
	&=\inf_{\eta_2\leq_{\text{cx}}\mu_2}\text{WC}(\mu_1,\eta_2)\\
	&=\inf_{\nu\in\mathcal P(\Omega)}\inf_{\eta_2\leq_{\text{cx}}\mu_2}tW_p^p(\eta_2,\nu)+(1-t)W_p^p(\nu,\mu_1).
\end{align*}
Therefore, a Wasserstein centroid $\bar\nu$ that minimizes $\text{WC}(\mu_1,\bar\eta_2)$ is also a weak centroid for $\text{wC}(\mu_1,\mu_2)$.
Using the properties of the Wasserstein centroid problem, we get an explicit descrition of a minimizer of \eqref{weakmidpointproblemN2} as a McCann interpolant \cite{mccann1997} between $\mu_1$ and $\bar\eta_2$:
\begin{lma}\label{dynamicalweakbarthm}
	Let $p\in(1,\infty)$, $\mu_1\in\mathcal P_{\text{ac}}(\Omega)$, $\mu_2\in\mathcal P(\Omega)$ and $t\in[0,1]$. Then a solution of \eqref{weakmidpointproblemN2} is given through the $p$-Wasserstein geodesic
	\begin{gather*}
		\bar\nu=\big((1-s(t))\text{id}+s(t)S\big)_{\#}\mu_1
	\end{gather*}
	where $S$ is the $p$-Wasserstein optimal map between $\mu_1$ and $\bar\eta_2$ and
	\begin{gather*}
		s(t)\coloneqq\frac{t^{\frac{1}{p-1}}}{t^{\frac{1}{p-1}}+(1-t)^{\frac{1}{p-1}}}.
	\end{gather*}
\end{lma}
\begin{proof}
		Lemma \ref{cMMW_plma} shows that the solution $\bar\nu$ of $\text{WC}(\mu_1,\bar\eta_2)$ is given by a $p$-Wasserstein geodesic between $\mu_1$ and $\bar\eta_2$, which is unique and has the required form \cite[Thm 5.27]{santambrogio2015} as $\mu_1\in\mathcal P_{\text{ac}}(\Omega)$. The considerations above show that $\bar\nu$ also solves $\text{wC}(\mu_1,\mu_2)$.
\end{proof}
Similarly, to how the solution to the Wasserstein centroid problem is also the solution to the Benamou-Brenier formula for $p$-Wasserstein \cite[6.2]{agueh2011}, we can now relate the solution of the weak centroid problem from Lemma \ref{dynamicalweakbarthm} to the dynamical version of the weak optimal transport, which was recently discovered by
Guo, Nilsson, and Wiesel \cite{guo2026} and Gozlan, Le Gouic, and Samson \cite{gozlan2025} for $p=2$:
\begin{thm}\cite[Thm 2]{guo2026}\cite[Thm 9]{gozlan2025}\label{guoweakbbthm}
	Let $\mu_1,\mu_2\in\mathcal P(\Omega)$. Then
	\begin{gather*}
		\mathcal T_2(\mu_2|\mu_1)=\inf_{(v_s,\sigma_s)_{s\in[0,1]}}\Big\{\mathbb{E}\Big[\int_0^1|v_s|^2\ \dd s\Big]:\ \dd X_s=v_s\dd s+\sigma_s\dd B_s, X_0\sim\mu_1,X_1\sim\mu_2\Big\}
	\end{gather*}
	where the infimum is taken over predictable processes $(v_s,\sigma_s)_{s\in[0,1]}$.
	The infimum is attained by the process
	\begin{gather*}
		\dd \bar X_s=(\nabla\bar\varphi(X_0)-X_0)\dd s+\bar\sigma_s\dd B_s
	\end{gather*}
	with $X_0\sim\mu_1$
	where $\nabla\bar\varphi$ is the $2$-Wasserstein optimal map between $\mu_1$ and $\bar\eta_2$ which solves $\inf\limits_{\eta_2\leq_{\text{cx}}\mu_2}W_2^2(\mu_1,\eta_2)$ and where $\bar\sigma_s\dd B_s$ is a stretched Brownian motion between $\bar\eta_2$ and $\mu_2$.
\end{thm}
Using Lemma \ref{dynamicalweakbarthm}, we recover the optimal process from Theorem \ref{guoweakbbthm} without diffusion term:
\begin{thm}\label{weakBBrelthm}
	Let $p=2$, $\mu_1\in\mathcal P_{\text{ac}}(\Omega)$, $\mu_2\in\mathcal P(\Omega)$ and $t\in[0,1]$.
	Then the solution $\bar\nu$ of $\text{wB}(\mu_1,\mu_2)$ given in Lemma \ref{dynamicalweakbarthm} is the law of $Y_t$, where the process $(Y_s)_{s\in[0,1]}$ fulfills $Y_0\sim\mu_1$, $Y_1\sim\bar\eta_2$ and $\dd Y_s=(\nabla\bar\varphi(Y_0)-Y_0)\dd s$ with $\bar\eta_2$ and $\nabla\bar\varphi$ from Theorem \ref{guoweakbbthm}.
\end{thm}
\begin{proof}
	By uniqueness of $\bar\eta_2$ the map $\nabla\bar\varphi$ from Theorem \ref{guoweakbbthm} is equal to $S$ from Lemma \ref{dynamicalweakbarthm}. Therefore, $\bar\nu$ is the law of $Y_t$, where the process $(Y_s)_{s\in[0,1]}$ fulfills $Y_0\sim\mu_1$, $Y_1\sim\bar\eta_2$ and $\dd Y_s=(\nabla\bar\varphi(Y_0)-Y_0)\dd s$.
\end{proof}
\section{Appendix}
\subsection{Two-marginal transport with infimal convolution cost and $p$-Wasserstein}
\begin{lma}\label{cMMW_plma}
	Let $p\in(1,\infty)$, $\mu_1,\eta_2\in\mathcal P(\Omega)$ and $t\in[0,1]$. Then
	\begin{gather*}
		\text{CMM}(\mu_1,\eta_2)=C(p,t)W_p^p(\mu_1,\eta_2)
	\end{gather*}
	with $\text{CMM}(\mu_1,\eta_2)$ from \eqref{CMM} and $C(p,t)\geq 0$.
\end{lma}
\begin{proof}
	Let $(\nu_s)_{s\in[0,1]}$ denote a $p$-Wasserstein geodesic between $\mu_1$ and $\eta_2$.
	We compute
	\begin{align*}
		\text{CMM}(\mu_1,\eta_2)&=\inf_{\nu\in \mathcal P(\Omega)}(1-t)W_p^p(\mu_1,\nu)+tW_p^p(\nu,\eta_2)\\
		&=\inf_{s\in[0,1]}(1-t)W_p^p(\mu_1,\nu_s)+tW_p^p(\nu_s,\eta_2)\\
		&=\inf_{s\in[0,1]}\big((1-t)s^p+t(1-s)^p\big)W_p^p(\mu_1,\eta_2).
	\end{align*}
	The minimum in $\inf_{s\in[0,1]}(1-t)s^p+t(1-s)^p$ is uniquely attained by
	\begin{gather*}
		s(t)\coloneqq\frac{t^{\frac{1}{p-1}}}{t^{\frac{1}{p-1}}+(1-t)^{\frac{1}{p-1}}}
	\end{gather*}
	and hence we get
	\begin{gather*}
		\text{CMM}(\mu_1,\eta_2)=\underbrace{\Bigg((1-t)\Big(\frac{t^{\frac{1}{p-1}}}{t^{\frac{1}{p-1}}+(1-t)^{\frac{1}{p-1}}}\Big)^p+t\Big(\frac{(1-t)^{\frac{1}{p-1}}}{t^{\frac{1}{p-1}}+(1-t)^{\frac{1}{p-1}}}\Big)^p\Bigg)}_{=:C(p,t)}W_p^p(\mu_1,\eta_2).
	\end{gather*}
\end{proof}
\subsection{Proof of Lemma \ref{dualattainmentlma}}
Before we prove Lemma \ref{dualattainmentlma}, we state a result from convex geometry, due to Csiszár and Matú\v{s} \cite{csiszar2001}.
\begin{definition}\cite[p 177]{csiszar2001}
	Let $\mu\in\mathcal P(\Omega)$. The set $\text{cc}(\mu)$ is the intersection of all convex sets $C\subset\RR^d$ such that $\mu(C)=1$. We call $\text{cc}(\mu)$ the \textit{convex core of $\mu$}.
\end{definition}
\begin{thm}\cite[Thm 3]{csiszar2001}\label{csiszarthm}
	Let $\mu\in\mathcal P(\RR^d)$ such that $\EE[\mu]$ exists and let $x\in\text{cc}(\mu)$. Then there exists $\lambda_x\in\mathcal P(\RR^d)$ such that $\lambda_x<<\mu$ and $\EE[\lambda_x]=x$ and such that there exists $L_x>0$ with $\|\frac{\dd\lambda_x}{\dd\mu}\|_{L^\infty(\mu)}<L_x$.
\end{thm}
We now give the proof of Lemma \ref{dualattainmentlma}, split into 5 steps:
\begin{proof}

	\setdefaultleftmargin{0pt}{}{}{}{}{}
	\begin{enumerate}
	\item \textit{Set-up:}
	Let $(\varphi_1^n,\dots,\varphi_N^n)_{n\in\NN}\subset\mathcal A_{\bar c}$ be a maximizing sequence for
	\begin{gather}\label{dualfunctional}
		(\varphi_1,\dots,\varphi_N)\mapsto \sum_{i=1}^M\int_\Omega\varphi_i(x_i)\ \dd\mu_i(x_i)+\sum_{i=M+1}^N\int_{K_i}\varphi_i(x_i)\ \dd\mu_i(x_i).
	\end{gather}
	\eqref{weakMMdual}, hence clearly the sequence is in $\mathcal B_{\bar c}$ and also maximizes \eqref{weakMMdualrelaxed}.
	We improve the value of \eqref{dualfunctional} by replacing for $j=1,\dots,M$
	\begin{gather*}
		\varphi_j^n(x_j)\rightarrow (\hat\varphi_j^{n})^{**}(x_j)\coloneqq\Big(\inf_{(x_i)_{i\neq j}}\bar c(x_1,\dots,x_N)-\sum_{i\neq j}\varphi_i^n(x_i)\Big)^{**}(x_j)
	\end{gather*}
	for $j=1,\dots,M$. Here, $f^{**}$ denotes the biconjugate which has the property that it fulfills $f^{**}\leq f$ and is the largest convex function to do so (continuity of $(\hat\varphi_j^n)^{**}$ will follow later in the proof).\\
	We also normalize the first $N-1$ potentials through
	\begin{gather*}
		\varphi_1^n(\EE[\mu_1])=\dots=\varphi_{N-1}^n(\EE[\mu_{N-1}])=0
	\end{gather*}
	for all $n\in\NN$.
	\item \textit{Bounds for the convex potentials:}
	Fix a $j=1,\dots,M$ and write
	\begin{gather*}
		\inf_{(x_i)_{i\neq j}}\bar c(x_1,\dots,x_N)-\sum_{i\neq j}\varphi_i^n(x_i)=\inf_{z\in\Omega}\lambda_j|x_j-z|^p+\sum_{i\neq j}\inf_{x_i\in\Omega}\lambda_i|x_i-z|^p-\varphi_i^n(x_i).
	\end{gather*}
	Therefore, $\hat\varphi_j^n$ is given through an infimal convolution and hence (extending the functions in the infimal convolution to $+\infty$ outside of $\Omega$ to ensure the suprema are taken inside of $\Omega$)
	\begin{gather*}
		\varphi_j^n(x_j)=\sup_{\xi\in\RR^d}x_j\cdot\xi-\big(\lambda_j|\cdot|^p\big)^*(\xi)-\Big(\sum_{i\neq j}\inf_{x_i\in\Omega}\lambda_i|x_i-\cdot|^p-\varphi_i^n(x_i)\Big)^*(\xi).
	\end{gather*}
	By picking $\xi=0$ we get
	\begin{align*}
		\varphi_j^n(x_j)&\geq-\big(\lambda_j|\cdot|^p\big)^*(0)-\Big(\sum_{i\ne j}\inf_{x_i\in\Omega}\lambda_i|x_i-\cdot|^p-\varphi_i^n(x_i)\Big)^*(0)\\
		&=-\sup_{v_1\in\Omega}-\lambda_j|v_1|^p-\sup_{v_2\in\Omega}-\sum_{i\neq j}\inf_{x_i\in\Omega}\lambda_i|x_i-v_2|^p-\varphi_i^n(x_i)\\
		&=\inf_{v_1\in\Omega}\lambda_j|v_1|^p+\inf_{v_2\in\Omega}\sum_{i\neq j}\inf_{x_i\in\Omega}\lambda_i|x_i-v_2|^p-\varphi_i^n(x_i)\\
		&=\inf_{v_2\in\Omega}\sum_{i\neq j}\inf_{x_i\in\Omega}\lambda_i|x_i-v_2|^p-\varphi_i^n(x_i)=:C_1(n)
	\end{align*}
	and we also get
	\begin{align*}
		\varphi_j^n(x_j)&\leq\inf_{z\in\Omega}\lambda_j|x_j-z|^p+\sum_{i\neq j}\inf_{x_i\in\Omega}\lambda_i|x_i-z|^p-\varphi_i^n(x_i)\\
		&\leq \inf_{z\in\Omega}\lambda_j\text{diam}(\Omega)^p+\sum_{i\neq j}\inf_{x_i\in\Omega}\lambda_i|x_i-z|^p-\varphi_i^n(x_i)=:C+C_1(n)
	\end{align*}
	for some $C>0$ and we notice that $C_1(n)\in[-\infty,+\infty]$ does not depend on $x_j$ anymore.
	Because of $0=\varphi_j^n(\EE[\mu_j])$ we get
	\begin{gather*}
		C_1(n)\leq 0\leq C+C_1(n).
	\end{gather*}
	which shows equivalently that
	\begin{gather*}
		-C\leq C_1(n)\leq 0
	\end{gather*}
	and that therefore $-C\leq\varphi_j^n(x_j)\leq C$ with $C\in\RR$ independent of $n$ and $x_j$ (and it shows that replacing $\varphi_j^n$ by the bi-conjugate preserves continuity).
	\item \textit{Upper bound for the concave potentials:}
	For $j=M+1,\dots,N-1$ we get, using the reverse Jensen inequality,
	\begin{gather}\label{integralbound}
		0=\varphi_j^n(\EE[\mu_j])\geq\int_{\Omega}\varphi_j^n(x_j)\ \dd\mu_j(x_j).
	\end{gather}
	Because $(\varphi_1^n,\dots,\varphi_N^n)_{n\in\NN}$ are a maximizing sequence for \eqref{weakMMdual}, we get a lower bound
	\begin{gather}\label{lowerboundmaximization}
		C\leq\sum_{i=1}^N\int_{\Omega}\varphi_i^n(x_i)\ \dd\mu_i(x_i).
	\end{gather}
	Recalling that we have
	\begin{gather*}
		\varphi_j^n(x_j)\leq\bar c(x_1,\dots,x_N)-\sum_{i\neq j}\varphi_i^n(x_i)
	\end{gather*}
	and integrating with respect to the product measure
	\begin{gather*}
		\gamma_{-j}\coloneqq\mu_1\otimes\dots\mu_{j-1}\otimes\mu_{j+1}\otimes\dots\otimes\mu_N
	\end{gather*} we get a uniform upper bound for $\varphi_j^n$ through
	\begin{align*}
		\varphi_j^n(x_j)&\leq\int_{\Omega^{d-1}}\bar c(x_1,\dots,x_N)\ \dd\gamma_{-j}-\sum_{i\neq j}\int_{\Omega}\varphi_i^n(x_i)\ \dd\mu_i(x_i)\\
		&\leq\int_{\Omega^{d-1}}\bar c(x_1,\dots,x_N)\ \dd\gamma_{-j}-\sum_{i=1}^N\int_{\Omega}\varphi_i^n(x_i)\ \dd\mu_i(x_i)\\
		&\leq C
	\end{align*}
	where in the second line we have used \eqref{integralbound} and in the third line we have used the boundedness of $\bar c$ over $\Omega$ and \eqref{lowerboundmaximization}.\\
	For $j=N$ (the unnormalized concave $\varphi_N^n$) we get the upper bound
	\begin{gather*}
		\varphi_N^n(x_N)\leq\bar c(\EE[\mu_1],\dots,\EE[\mu_{N-1}],x_N)\leq C
	\end{gather*}
	using again the boundedness of $\bar c$ over $\Omega$.
	\item \textit{Lower bound for the concave potentials:}
	Fix $j=M+1,\dots,N$ and recall that $K_j=\text{conv}(\text{supp}(\mu_j))$ is a compact set. Fix $C_j\Subset \text{int}(K_j)$. We extend the pointwise bound $L_x$ from Theorem \ref{csiszarthm} to a uniform bound $L_{C_j}$ on $C_j$: We have $\overline{\text{cc}(\mu_j)}=K_j$ by \cite[Lma 1]{csiszar2001} and hence $C_j\subset\text{int}(K_j)=\text{int}(\overline{\text{cc}(\mu_j)})=\text{int}(\text{cc}(\mu_j))\subset\text{cc}(\mu_j)$. Therefore, by Theorem \ref{csiszarthm} for each $x\in C_j$ there exists $\lambda_x\in\mathcal P(\RR^d)$ such that $\lambda_x<<\mu_j$ and $\EE[\lambda_x]=x$ and such that there exists $L_x>0$ with $\|\frac{\dd\lambda_x}{\dd\mu_j}\|_{L^\infty(\mu_j)}<L_x$.\\ We now redo the proof of \cite[Thm 3]{csiszar2001} to show that $L_x$ can be chose uniformly for all $x\in C_j$:
	We fix $x_0\in C_j$ and because $x_0\in\text{int}(K_j)$ we find points $a_1,\dots,a_d\in\text{int}(K_j)$ such that $x_0\in\text{int}(\text{conv}(a_1,\dots,a_d))$. For every $x\in\text{conv}(a_1,\dots,a_d)$, there exist $(\alpha_l)_{l\in[d]}\subset[0,1]^d$ such that $\sum_{l=1}^d\alpha_l=1$ and such that
	\begin{gather*}
		x=\sum_{l=1}^d\alpha_la_l.
	\end{gather*}
	We define the probability measure
	\begin{gather*}
		\beta_{x}\coloneqq\sum_{l=1}^d\alpha_l\lambda_{a_l},
	\end{gather*}
	using the probability measures $\lambda_{a_l}$ for $l=1,\dots,d$ from Theorem \ref{csiszarthm}. Then,
	\begin{gather*}
		\EE[\beta_x]=\sum_{l=1}^d\alpha_l\EE[\lambda_{a_l}]=\sum_{l=1}^d\alpha_la_l=x,
	\end{gather*}
	$\beta_x<<\mu_j$ and
	\begin{gather*}
		\Big\|\frac{\dd\beta_x}{\dd\mu_j}\Big\|_{L^\infty(\mu_j)}=\Big\|\sum_{l=1}^d\alpha_l\frac{\dd\lambda_{a_l}}{\dd\mu_j}\Big\|_{L^\infty(\mu_j)}\leq \sum_{l=1}^d\alpha_lL_{a_l}\leq\max_{l=1,\dots,d}L_{a_l}.
	\end{gather*}
	Because $x_0\in\text{int}(\text{conv}(a_1,\dots,a_d))$ there exists a neighborhood $U_{x_0}\subset\text{conv}(a_1,\dots,a_m)$ and because $C_j$ is compact there is a countable number $x_1,\dots,x_O$ of points such that their associated $\text{conv}(a_1,\dots,a_m)$ cover are $C_j$. Therefore, there exists
	\begin{gather*}
		L_{C_j}\coloneqq \max_{r=1,\dots,O} L_{x_r}
	\end{gather*}
	such that for every $x\in C_j$ there exists a probability measure $\beta_x<<\mu_j$ with $\EE[\beta_x]=x$ and such that $\|\frac{\dd\beta_x}{\dd\mu_j}\|_{L^\infty(\mu_j)}\leq L_{C_j}$.\\
	From step (2), we know that for every $j=1,\dots,M$ we have
	\begin{gather*}
		C\leq \int_\Omega\varphi_j^n(x_j)\ \dd\mu_j(x_j)\leq C.
	\end{gather*}
	For every $j=M+1,\dots,N$ we have due to step (3) that 
	\begin{gather*}
		\int_{\Omega}\varphi_j^n(x_j)\ \dd\mu_j(x_j)\leq C
	\end{gather*}
	and hence because $(\varphi_1^n,\dots,\varphi_N^n)_{n\in\NN}$ is a maximizing sequence of \eqref{dualfunctional} we also get for every $j=M+1,\dots,N$ that 
	\begin{gather*}
		\int_{K_j}\varphi_j^n(x_j)\ \dd\mu_j(x_j)\geq C.
	\end{gather*}
	Now for each $x\in C_j$ we get from Jensen's inequality and the constructed $\beta_x$ that
	\begin{gather*}
		\varphi_j^n(x)=\varphi_j^n(\EE[\beta_x])\geq \int_\Omega \varphi_j^n(y)\ \dd\beta_x(y)=\int_\Omega \varphi_j^n(y) \frac{\dd\beta_x}{\dd\mu_j}(y)\ \dd\mu_j(y).
	\end{gather*}
	Recalling that in step (3) we got $\varphi_j^n(x_j)\leq A$, with $A$ independent of $x_j$ and $n$, we get the lower bound
	\begin{align*}
		\varphi_j^n(x_j)&\geq A-\int_\Omega \big(A-\varphi_j^n(y)\big) \frac{\dd\beta_x}{\dd\mu_j}(y)\ \dd\mu_j(y)\\
		&\geq A-\Big\|\frac{\dd\beta_x}{\dd\mu_j}\Big\|_{L^\infty(\mu_j)}\int_\Omega \big(A-\varphi_j^n(y)\big) \ \dd\mu_j(y)\\
		&\geq A-L_{C_j}(A-C)
	\end{align*}
	which is independent of $x_j$ and $n$ (but still depends on $C_j$).
	\item \textit{Passing to the limit:}
	Due to the equiboundedness on the whole of $\Omega$ that we found in step (2) for each $j=1,\dots,M$, we can apply Arzela-Ascoli to each $(\varphi_j^n)_{n\in\NN}$ to get a convergent subsequence and using a diagonalization argument we find a subsequence of  $(\varphi_1^n,\dots,\varphi_N^n)_{n\in\NN}$ (again denoted by $(\varphi_1^n,\dots,\varphi_N^n)_{n\in\NN}$), such that for each $j=1,\dots,M$ and for each $x_j\in\Omega$ we have $\varphi_j^n(x_j)\underset{n\rightarrow\infty}{\longrightarrow}\bar\varphi_j(x_j)$ where $\bar\varphi_j$ is convex.\\
	For $j=M+1,\dots,N$ we get equiboundeness of $(\varphi_j^n)_{n\in\NN}$ for every compact $C_j\subset\text{int}(K_j)$ and using a diagonalization argument we find a subsequence of $(\varphi_j^n)_{n\in\NN}$ (again denoted by $(\varphi_j^n)_{n\in\NN}$) such that $\varphi_j^n(x_j)\underset{n\rightarrow\infty}{\longrightarrow}\bar\varphi_j(x_j)\in\RR$ for every $x_j\in\text{int}(K_j)$. Using another diagonalization argument over $j=M+1,\dots,N$, we finally arrive at a subsequence $(\varphi_1^n,\dots,\varphi_N^n)_{n\in\NN}$ that converges pointwise (for the convex potentials on the whole of $\Omega$ and for the concave potentials only on $\text{int}(K_j)$) to $(\bar\varphi_1,\dots,\bar\varphi_N)$. To extend the concave $\bar\varphi_j$ with $j=M+1,\dots,N$ to all of $\Omega$, we replace them with
	\begin{gather*}
		\hat\varphi_j(x_j)\coloneqq\underset{\substack{y_j\rightarrow x_j\\ y_j\in\text{int}(K_j)}}{\lim\sup}\bar\varphi_j(y_j)\in[-\infty,C].
	\end{gather*}
	Because $\bar\varphi_j$ is continuous on $\text{int}(K_j)$, we have for every $x_j\in \text{int}(K_j)$
	\begin{gather*}
		\hat\varphi_j(x_j)=\bar\varphi_j(x_j)
	\end{gather*}
	and additionally $\hat\varphi_j$ is upper-semicontinuous and concave (but it might take the value $-\infty$ on the boundary of $K_j$).\\
	Let $(x_1,\dots,x_N)\in\Omega^M\times K_{M+1}\times\dots\times K_N$ and let for each $j=M+1,\dots,N$ $(y_j^k)_{k\in\NN}\subset K_j$ be a sequence realizing the $\lim\sup$ in $\hat\varphi_j(x_j)$.
	Then we get
	\begin{align*}
		\sum_{i=1}^M\bar\varphi_i(x_i)+\sum_{i=M+1}^N\hat\varphi_i(x_i)&=\lim_{n\rightarrow\infty}\lim_{k\rightarrow\infty}\sum_{i=1}^M\varphi_i^n(x_i)+\sum_{i=M+1}^N\varphi_i^n(y_i^k)\\
		&\leq \lim_{n\rightarrow\infty}\lim_{k\rightarrow\infty}\bar c(x_1,\dots,x_M,y_{M+1}^k,\dots,y_{N}^k)\\
		&=\bar c(x_1,\dots,x_N)
	\end{align*}
	and hence $(\bar\varphi_1,\dots,\bar\varphi_M,\hat\varphi_{M+1},\dots,\hat\varphi_{N})\in\mathcal B_{\bar c}$.\\
	Applying Fatou's lemma (notice that the upper bounds found in steps (2) and (3) are integrable and also apply to $\hat\varphi_j$) we get
	\begin{align*}
		\text{wDr}(\mu_1,\dots,\mu_N)
		&\leq\underset{n\rightarrow\infty}{\lim\sup}\sum_{i=1}^M\int_\Omega\varphi_i^n(x_i)\ \dd\mu_i(x_i)+\sum_{i=M+1}^N\int_{K_i}\varphi_i^n(x_i)\ \dd\mu_i(x_i)\\
		&\leq\sum_{i=1}^M\int_\Omega\underset{n\rightarrow\infty}{\lim\sup}\ \varphi_i^n(x_i) \dd\mu_i(x_i)+\sum_{i=M+1}^N\int_{K_i}\underset{n\rightarrow\infty}{\lim\sup}\ \varphi_i^n(x_i) \dd\mu_i(x_i)\\
		&=\sum_{i=1}^M\int_\Omega\bar\varphi_i(x_i)\ \dd\mu_i(x_i)+\sum_{i=M+1}^N\int_{K_i}\bar\varphi_i(x_i)\ \dd\mu_i(x_i)\\
		&=\sum_{i=1}^M\int_\Omega\bar\varphi_i(x_i)\ \dd\mu_i(x_i)+\sum_{i=M+1}^N\int_{K_i}\hat\varphi_i(x_i)\ \dd\mu_i(x_i)
	\end{align*}
	which proves that \eqref{weakMMdualrelaxed} admits maximizers.
	\end{enumerate}
\end{proof}

\section*{Acknowledgements}
The author is grateful to Robert McCann and Adrian Nachman for their supervision and ongoing support while writing this manuscript. The autor further thanks Yair Shenfeld for suggesting the topic of the present manuscript to him.

\section*{Tool and computational resource disclosure}
The author used ChatGPT, provided by OpenAI, to brainstorm approaches to mathematical arguments and identify relevant literature. The tools were used to generate suggestions for subsequent evaluation, not as sources of mathematical or bibliographic authority. Especially, the relaxed problem \eqref{weakMMdualrelaxed}, the use of the biconjugate in the proof of Lemma \ref{dualattainmentlma} and the argument for step (4) in this proof were suggested by ChatGPT.
The author independently checked every mathematical argument and verified every citation against the cited source without
relying on artificial intelligence. The author takes full responsibility for the accuracy, originality, and integrity of the manuscript

\bibliographystyle{plain}
\bibliography{bibliography}

\end{document}